\documentclass[12pt]{amsart}

\usepackage{amsmath,amsfonts,amssymb,amsthm}
\usepackage{tikz}
\usepackage[left=2.9cm,top=2.5cm,right=2.9cm,bottom=2.6cm]{geometry}

\allowdisplaybreaks

\theoremstyle{definition}
\newtheorem{theorem}{Theorem}

\newtheorem{lemma}[theorem]{Lemma}
\newtheorem{corollary}[theorem]{Corollary}

\newtheorem{definition}[theorem]{Definition}

\newtheorem*{thm-fast}{Theorem~\ref{thm:fast-growth}}
\newtheorem*{thm-linear}{Theorem~\ref{thm:growth-by-one}}
\newtheorem*{thm-rootn}{Theorem~\ref{thm:growth-by-rootn}}

\title{Burning the plane: densities of the infinite Cartesian grid}
\author{Anthony Bonato}
\address{Department of Mathematics, Ryerson University, Toronto ON, Canada, M5B 2K3}
\email[Anthony Bonato]{abonato@ryerson.ca}
\thanks{AB and KG supported by NSERC Discovery grants.}
\author{Karen Gunderson}
\address[Karen Gunderson and Amy Shaw]{Department of Mathematics, University of Manitoba, Winnipeg MB, Canada, R3T 2N2}
\email[Karen Gunderson]{karen.gunderson@umanitoba.ca}
\author{Amy Shaw}
\email[Amy Shaw]{amy@amyshaw.email}

\begin{document}

\maketitle

\begin{abstract}
Graph burning is a discrete-time process on graphs, where vertices are sequentially burned, and burned vertices cause their neighbours to burn over time. We consider extremal properties of this process in the new setting where the underlying graph is also changing at each time-step.  The main focus is on the possible densities of burning vertices when the sequence of underlying graphs are growing grids in the Cartesian plane, centred at the origin.  If the grids are of height and width $2cn+1$ at time $n$, then all values in $\left [ \frac{1}{2c^2} , 1 \right ]$ are possible densities for the burned set.  For faster growing grids, we show that there is a threshold behaviour: if the size of the grids at time $n$ is $\omega(n^{3/2})$, then the density of burned vertices is always $0$, while if the grid sizes are $\Theta(n^{3/2})$, then positive densities are possible. Some extensions to lattices of arbitrary but fixed dimension are also considered.

\end{abstract}

\section{Introduction}\label{sec:intro}

Numerous recent works have analyzed the spread of social contagion in real-word networks. As an example, \cite{Kramer14} demonstrated that emotional states can be transferred to others on Facebook without direct interaction between people and in the complete absence of nonverbal cues. Graph burning is a new, discrete-time process that measures how prone a network is to fast social contagion. The input of the process is an undirected finite graph and at each step, vertices are either \emph{burning} (also called \emph{burned}) or not. Initially, a single vertex, called an \emph{activator}, is burning and in each subsequent round, every neighbour of a burning vertex becomes burned and a fire is ignited at a new activator which is also burning. The process is complete once all vertices are burning. In this paper, the focus is on the behaviour of the process when the sequence of activators is chosen deterministically to minimize the number of rounds.  The \emph{burning number} of a graph $G$ is the minimum number of rounds it takes for all vertices to be burned in $G$.  

Bonato et al.\ \cite{BJR14,BJR16,thez} first introduced the burning process, found bounds, and characterized the burning number for various graph classes. It was proved by Bessy et al.\ in \cite{BBJRR18} that for a connected graph of order $n$, the burning number is at most $2\lceil\sqrt{n}\rceil-1$, and this bound was improved by Land and Lu \cite{LL} to $\frac{\sqrt{6}}{2} \sqrt{n}$. In \cite{BBJRR18}, Bessy et al.\ conjectured that the burning number of a connected graph of order $n$ is at most $\lceil \sqrt{n} \rceil$, and this was shown to hold for spider graphs in \cite{BL} by Bonato and Lidbetter. Mitsche, Pra{\l}at, and Roshanbin \cite{rburn} considered randomized burning on the path and Fitzpatrick and Wilm \cite{shannon} considered burning for circulants. It was proved by Bessy et al.\ in \cite{BBJRR17} that it is \textbf{NP}-hard to determine the burning number even in elementary graph families such as trees with maximum degree three, spider graphs, and forests consisting of disjoint unions of paths. Approximation algorithms for graph burning were given by Bonato and Kamali \cite{BK}.

The process of graph burning finds a place among several models that measure the spread of a fire or contagion against time. One process in which a fire spreads to neighbours of burning vertices, is firefighting on graphs. Introduced by Hartnell \cite{hartnell} in 1995, the goal of firefighting is to extinguish or limit the growth of a `fire' that breaks out at one or several vertices in a graph by protecting a limited number of vertices at each time step. See Finbow and MacGillivray \cite{finbow} for a survey of firefighting results.
A notion of density for firefighting, called the surviving rate, was introduced by Cai and Wang \cite{CW}.

Another model measuring the spread of an activation or `infection' in a graph is the $r$-\emph{neighbour bootstrap process}, a discrete-time process where a vertex becomes infected when at least $r$ of its neighbours are.  Recently, extremal questions regarding the time to full infection have been considered by Benevides and Przykucki~\cite{benevides1,benevides2} and Przykucki \cite{przykucki}.

There are a number of further models for a randomized spread of infection in a graph.  These models often arise in the context of diffusion or rumour-spreading in social networks. These include Hammersley and Welsh's first-passage percolation model~\cite{HW65}, Richardson's model for the spread of disease~\cite{dR73}, Harris's contact process~\cite{tH74}, diffusion models given by Granovetter~\cite{granovetter} and Schelling~\cite{schelling}, and randomized push\&pull gossip algorithms given by Boyd, Ghosh, Prabhakar, and Shah~\cite{BGPS06}.

In this paper, we consider a variant of the burning process, where the underlying graph is also growing at each step.  In the case that the underlying graph becomes arbitrarily large, it need not be the case that there is ever a particular step in the burning process where all vertices of the current graph are burned.  We may ask instead about the behaviour of the \emph{proportion} of burning vertices. The precise statement of this new burning process is given below in Definition~\ref{def:burn-seq}.  Included in the definition is the possibility that no new vertex is activated at a particular time step.  Having the ability to `skip a chance' to activate a new vertex is particularly useful for analysis of an infinite sequence of growing graphs in which burning sets in sequences of disjoint subgraphs are considered separately.

\begin{definition}\label{def:burn-seq}
Let $V$ be an infinite set and let $\underline{G} = (G_0, G_1, \dots)$ be a sequence of graphs with the property that for every $n \geq 1$, $V(G_n) \subset V$ and $G_{n-1}$ is an induced subgraph of $G_n$.  Let $\underline{v} \in \prod_{n \geq 0} (V(G_n) \cup \{\bullet\})$.  The sequence $\underline{v}$ is the sequence of activator vertices. The notation $v_n = \bullet$ indicates that no vertex was activated.    Define the sequences of burning sets by $B_0 = \{v_0\},$ and for every $n \geq 0$, define
\[
B_{n+1} =
\begin{cases}
	N_{G_{n+1}}[B_n]	&\text{if $v_{n+1} = \bullet$},\\
	N_{G_{n+1}}[B_n] \cup \{v_{n+1}\}			&\text{otherwise}.
\end{cases}
\]
The sequence $\underline{v}$ is \emph{valid} if for all $n \geq 0$ with $v_{n+1} \neq \bullet$, $v_{n+1} \in V(G_{n+1}) \setminus N_{G_{n+1}}[B_{n}]$.  In a valid sequence of activated vertices, at each time step, the newly activated vertex is not already burned by a vertex from the previous time step.

The proportion of vertices burning at time $n$ is then $|B_n|/|V(G_n)|$ and the \emph{burning density} of a sequence $\underline{v}$ in $\underline{G}$ is defined to be
\[
\delta(\underline{G},\underline{v}) = \lim_{n \to \infty} \frac{|B_n|}{|V(G_n)|},
\]
if this limit exists.
As the limit above need not exist, the \emph{lower burning density} is defined to be
\[
\underline{\delta} (\underline{G},\underline{v}) = \liminf_{n \to \infty} \frac{|B_n|}{|V(G_n)|} .
\]
Similarly, the \emph{upper burning density} is defined to be
\[
\overline{\delta} (\underline{G},\underline{v}) = \limsup_{n \to \infty} \frac{|B_n|}{|V(G_n)|} .
\]
\end{definition}

Note that in the traditional version of the burning process, a new unburned vertex must be activated at each step.  The possibility of having $v_n = \bullet$ in a burning sequence corresponds to either allowing a `pass' for vertex activation, or allowing the possibility of activating an already-burning vertex.

The main focus here is on graphs that are square grids in the integer lattice.  Unless otherwise specified, $[a, b]\times [c, d]$ is used to denote the graph with vertex set $[a, b] \times [c,d]$ and with vertices at $L_1$-distance $1$ joined by an edge.  The corresponding notation is used for the $d$-dimensional lattice graphs.

Let $\mathbb{N}_0$ denote the natural numbers including $0$. The classes of sequences of grids studied here are of the following form: Given an increasing function $f: \mathbb{N}_0 \to \mathbb{N}_0$, consider the sequence of graphs given by $(\left[-f(n),f(n) \right]^2)_{n \geq 0}$.  The main question is: Given the function $f$, which real numbers are achievable as the burning density of some sequence of activated vertices that contains no $\bullet$; that is, no skipped step?  Also, for which functions is the burning density always $0$ and for which is a positive density possible?

In the case that $f(n)$ grows linearly, Theorem~\ref{thm:growth-by-c} gives an interval of all achievable burning densities.

\begin{theorem}\label{thm:growth-by-c}
Let $c\geq 1$ and let $\underline{S}$ be the sequence of square grids with $S_n = [-\lceil cn \rceil, \lceil cn \rceil]^2$. For each $\rho \in [\frac{1}{2c^2},1]$, there is a valid sequence of activated vertices $\underline{v}=(v_1,v_2,\dots) \in \prod_{n \geq 0} V(S_n)$ such that the burning density in $\underline{S}$ is $\delta (\underline{S},\underline{v})=\rho$.
\end{theorem}

The proof of Theorem~\ref{thm:growth-by-c} is given in Section~\ref{sec:growth-by-one}.  The proof proceeds by first giving a proof for the case $c = 1$ that involves considering a different sequence of graphs: $[0, n]^2$, determining the achievable burning densities in these, and then combining burning sequences for each of the four quadrants to achieve a given density in the larger grid.  The case $c \geq 1$ and $\rho = 1$ is then given in Theorem~\ref{thm:growth-by-c-density1}.  These results are combined to prove Theorem~\ref{thm:growth-by-c}.

Note that nothing is lost by restricting our attention to valid sequences of activated vertices. In the case $c = 1$, unless $\underline{v} = ( \bullet )_{i=1}^\infty$, then $\underline{\delta} ( \underline{S} , \underline{v} ) \geq 1/2$. Indeed, if the first vertex $u$ is activated at time $k$, then $B_n$ contains every vertex at distance at most $n-k$ from $u$, so that
\[
\underline{\delta} ( \underline{S} , \underline{v} ) \geq \lim_{n \to \infty} \frac{2(n-k)^2 + 2(n-k) + 1}{(2n+1)^2} = \frac{1}{2} .
\]

 For faster-growing grids, Theorem~\ref{thm:growth-by-rootn} shows that if there is a constant $c$ with $f(n)=\lceil c n^{3/2} \rceil$, then a positive lower burning density is possible, while Corollary~\ref{cor:fast-growth2} shows that in contrast, when $f(n)= \omega (n^{3/2})$, the burning density is always $0$.

\begin{theorem}\label{thm:growth-by-rootn}
For any positive constant $c$, let $\underline{S}$ be the sequence of square grids such that $S_n = \left [ - \lceil cn^{3/2} \rceil, \lceil cn^{3/2} \rceil \right ]^2$. Then there is a sequence of activated vertices $\underline{v}$ with lower burning density $\underline{\delta} (\underline{S},\underline{v}) >0$.
\end{theorem}

It is a relatively straightforward consequence of Theorem~\ref{thm:growth-by-rootn} that if $f:\mathbb{N}_0 \to \mathbb{N}_0$ is a function with $f(n) = \Theta(n^{3/2})$ then there exists a valid sequence of activated vertices $\underline{v}$ with lower burning density $\underline{\delta} (( [-f(n), f(n)]^2)_{n \geq 0}, \underline{v}) > 0$.

For faster growing grids, the result is stated for arbitrary dimensions, as the proof is the same for any fixed dimension.

\begin{theorem}\label{thm:fast-growth}
For any $d \geq 2$, let $f:\mathbb{N}_0 \to \mathbb{N}_0$ be a function with $f(n) = \omega(n^{(d+1)/d})$.  For every $n \geq 0$, let $S_n = [-f(n), f(n)]^d$, and let $\underline{S} = (S_n)_{n \geq 0}$.  Then, for every sequence of activated vertices $\underline{v}$ the burning density is $\delta(\underline{S}, \underline{v}) = 0$.
\end{theorem}

In the case $d =2$, Theorem~\ref{thm:fast-growth} gives the following result in contrast to Theorem~\ref{thm:growth-by-rootn}.

\begin{corollary}\label{cor:fast-growth2}
Let $f:\mathbb{N}_0 \to \mathbb{N}_0$ be a function with $f(n) = \omega(n^{3/2})$, for every $n \geq 0$, let $S_n = [-f(n), f(n)]^2$, and let $\underline{S} = (S_n)_{n \geq 0}$.  For every sequence of activated vertices $\underline{v}$ the burning density is $\delta(\underline{S}, \underline{v}) = 0$.
\end{corollary}

The proof of Theorem~\ref{thm:growth-by-rootn} is given in Section~\ref{sec:fast-growth}. In Section~\ref{sec:d-dim}, Theorem~\ref{thm:fast-growth} is proved and results generalizing Theorem~\ref{thm:growth-by-rootn} are given for grids of arbitrary, but fixed, dimension.  In particular, in Theorem~\ref{thm:d-dim-pos-dens}, it is shown that for any fixed dimension $d$, the graph sequence $(S_n)_{n \geq 1}$ given by $S_n = [-\lceil n^{(d+1)/d}\rceil, \lceil n^{(d+1)/d}\rceil]^d$ always has a sequence of activated vertices with positive lower burning density.  While the proof of Theorem~\ref{thm:growth-by-rootn} is constructive, the proof of Theorem~\ref{thm:d-dim-pos-dens} is probabilistic.

As distances within the lattice are used regularly within the proofs, for every $x \in \mathbb{Z}^d$ and $r \geq 0$, let $B_1(x, r)$ be the closed $L_1$-ball of radius $r$ centred at $x$: $B_1(x, r) = \{y \in \mathbb{Z}^d :\ ||x-y||_1 \leq r\}$.  For any $x, y \in \mathbb{Z}^d$, $d(x, y)$ is used to denoted the graph distance between $x$ and $y$, which is the same as the $L_1$-distance, denoted $d_1(x, y)$.  The $L_2$-distance between two points is denoted $d_2(x, y)$.

\section{Slowly growing grids}\label{sec:growth-by-one}

The results in this section address the case that for some $c \geq 1$, the sequence of underlying graphs is $([-cn, cn]^2)_{n \geq 0}$ and a vertex is activated at every time step, but we first examine the case in which $c=1$.  Within this framework, which burning densities are achievable?  Since the vertex $(0, 0)$ will always be activated, the ball of radius $n$ about the origin in $[-n, n]^2$ is always a subset of the burned vertices at time $n$.  Because of this central $L_1$ ball that is burned, if the burning density of a sequence of activated vertices exists, the density is in the real interval $[1/2, 1]$.  The first aim of this section is to show that, in fact, each of these values is achievable as a burning density.

One feature of having the ball of radius $n$ about the origin in $[-n, n]^2$ being burned is that the effect of newly activated vertices on future burned sets is only seen in the quadrant in which that activated vertex was chosen.  For this reason, we first determine the achievable burning densities in a single quadrant: the sequence of graphs $(Q_n)_{n \geq 0} = ([0, n]^2)_{n \geq 0}$.  In Lemma~\ref{lem:density0}, it is shown that, as long as no vertex is activated in at least one of the time steps $n =1$ or $n=2$, then burning density $1/2$ in $(Q_n)_{n \geq 0}$ is achievable.  In Lemma~\ref{lem:skinny-triangle}, for any $\rho \in (1/2, 1)$, a sequence of activated vertices is defined that achieves burning density $\rho$ in $(Q_n)_{n \geq 0}$.  Finally, in Lemma~\ref{lem:layer-cake}, it is shown that as long as there are bounded gaps between times that vertices are activated, a sequence can be chosen in $Q_n$ with burning density $1$.  Together, these three regimes of vertex-activation can be combined to produce sequences of activated vertices achieving any density in $[1/2, 1]$ for the sequence of graphs $([-n, n]^2)_{n \geq 0}$.  These combinations are described in the proof of Theorem~\ref{thm:growth-by-one}.

\begin{lemma}\label{lem:density0}
Let $\underline{Q}$ be the sequence of square grids such that $Q_n = [0,n]^2$. Define $\underline{v}$ to be a valid sequence of activated vertices such that either $v_1 = \bullet$ or $v_2 = \bullet$, and such that for any $n$ with $v_n \not = \bullet$, $v_n \in V(Q_n)$ is the unburned vertex nearest to $(0,0)$ with the greatest $y$-coordinate. Then we have that
\[
|B_n| \leq \frac{(n+2)(n+1)}{2} + (2+ \log_2 n)(n+1).
\]
\end{lemma}

\begin{proof}
The purpose of choosing vertices to activate as close as possible to $(0, 0)$ and with the greatest $y$-coordinate is to ensure that the burning set is contained in a ball centred at the origin with a radius that is growing only slightly faster than the growth-rate of the underlying grid.  For every $t \geq 0$, and $n \geq 1$, call the set of vertices whose cartesian coordinates satisfy $x+y = n+t$ and $0 \leq x,y \leq n$, the \emph{$t$-th diagonal at time $n$}.  Within the grid, the line $x+y = n$ is the boundary of the ball of radius $n$ around the origin and will always be burning at time $n$.  This is the $0$-th diagonal at time $n$.  The way in which the activated vertices are chosen will result in a burning set with the property that for some $t \geq 1$, $B_1((0, 0), n+t-1)$ is burning and the remaining burned vertices will be an initial sequence of the vertices in the $t$-th diagonal: $(t, n), (t+1, n-1), (t+2, n-2), \ldots, (n, t)$, ordered by increasing $x$-coordinates (or equivalently, by decreasing $y$-coordinates).

At time $n$, there are $n-t+1$ vertices in the $t$-th diagonal.  Suppose that the first $i < n-t+1$ of these vertices are burned at time $n$, then if, in the next step, the next vertex on the $t$-th diagonal is activated, then at step $n+1$, the `new' $t$-th diagonal has $(n+1)-t+1$ vertices, of which, $i+2$ are burned.  Continuing in this way, at time $n + (n-t+1-i)$, every vertex in the $t$-th diagonal is burned.

If either $v_1 = \bullet$ or $v_2 = \bullet$, then at time $n = 3$, the burned set is contained in the vertices of the grid whose cartesian coordinates satisfy $x+y \leq 4$.  In particular, there are no burned vertices in the $2$nd diagonal at time $3$.  For every $t \geq 2$, let $n_t$ be the largest $n$ for which the $t$-th diagonal has no burned vertices at time $n$.  We then have that $n_2 \geq 3$.  Since the $t$-th diagonal at time $n_t$ contains $n_t - t + 1$ vertices and will be entirely burned after at least $n_t-t+1$ time steps, $n_{t+1} \geq 2n_t-t+1$. Hence, we find that $n_t \geq (n_2-2)2^{t-2} + t$.  For any $t$ and $n$ with $n_{t} \leq n < n_{t+1}$, $B_n \subseteq B_1((0, 0), n+t)$ and
\begin{align*}
|B_1((0, 0), n+t) \cap [0, n]^2|
	&= |B_1((0, 0), n) \cap [0, n]^2| + \sum_{i = n-t+1}^{n} i\\
	&= \frac{(n+1)(n+2)}{2} + t(n+1) - \frac{t(t+1)}{2}\\
	& \leq  \frac{(n+1)(n+2)}{2} + t(n+1).
\end{align*}
If $n \geq n_t \geq 2^{t-2} = \frac{1}{4}2^t$, then $t \leq 2 + \log_2(n)$ so that
\[
|B_n| \leq  \frac{(n+1)(n+2)}{2} + (2 + \log_2(n))(n+1),
\]
and the proof follows.
 \end{proof}

We need the following lemma.

\begin{lemma}\label{lem:skinny-triangle}
Let $\underline{Q}$ be the sequence of square grids such that $Q_n = [0,n]^2$ and let $\rho \in (0,1)$. Let $\underline{v}$ be the sequence of activated vertices defined by $v_0 = (0,0)$ and for every $n \geq 1$,
\[
v_n =
\begin{cases}
(\lfloor \rho n \rfloor ,n) & \text{if } \lfloor \rho n \rfloor - \lfloor \rho(n-1) \rfloor = 1, \\
\bullet & \text{if } \lfloor \rho n \rfloor  - \lfloor \rho(n-1) \rfloor = 0.
\end{cases}
\]
We then have that $\delta (\underline{Q},\underline{v}) = \frac{1}{2} (1+ \rho)$.
\end{lemma}

\begin{proof}

For the choice of activators defined in the statement of the theorem, a vertex is activated approximately every $\frac{1}{\rho}$ steps. Each time a new vertex is activated, it is always chosen at the top edge of the grid at that time step and adjacent to a burned vertex.  For every $k \geq 1$, the top vertex in the $k$-th diagonal is activated at time $n$ with $n$ defined by $\rho(n-1) < k \leq \rho n$.  That is, the vertex $(k, n)$ is activated at time $n = \left\lceil \frac{k}{\rho} \right\rceil$.  For every $n \geq  \left\lceil \frac{k}{\rho} \right\rceil$, the number of burned vertices in the $k$-th diagonal at time $n$ is $n - \left\lceil \frac{k}{\rho} \right\rceil + 1$.

As the first vertex burned is the origin, for any $n$, all vertices in the ball of radius $n$ about the origin are burned at time $n$.  In the graph $Q_n$, this set of vertices has size $\frac{(n+2)(n+1)}{2}$.  Furthermore, for every $k \leq n \rho$, there are $n-\lceil k/\rho \rceil + 1$ vertices burned along the $k$-th diagonal at time $n$.  Thus,
\begin{align}
|B_n| &= \frac{(n+2)(n+1)}{2} + \sum_{k=1}^{\lfloor n \rho \rfloor} \left( n + 1 - \lceil k \rho^{-1} \rceil \right) \notag\\
 	&= \frac{(n+2)(n+1)}{2} + \lfloor n\rho \rfloor (n+1) - \sum_{k = 1}^{\lfloor n \rho \rfloor} \left\lceil \frac{k}{\rho} \right\rceil. \label{eq:burned-skinny}
\end{align}
To bound the size of the burned set, note that the sum in the expression in~\eqref{eq:burned-skinny} is bounded above by
\begin{align}
\sum_{k = 1}^{\lfloor n \rho \rfloor} \left\lceil \frac{k}{\rho} \right\rceil
	&\leq \sum_{k = 1}^{\lfloor n \rho \rfloor} \frac{k+\rho}{\rho}  \notag\\
	&=\frac{1}{\rho} \frac{\lfloor n \rho \rfloor(\lfloor n \rho \rfloor +1)}{2} + \lfloor n \rho \rfloor \notag\\
	&\leq \frac{n\rho(n\rho +1)}{2 \rho} + n \rho \notag\\
	&= \frac{n^2 \rho}{2} \left(1 + \frac{1}{n \rho} + \frac{2}{n}\right). \label{eq:ub-sum-skinny}
\end{align}
The same sum is bounded below by
\begin{align}
\sum_{k = 1}^{\lfloor n \rho \rfloor} \left\lceil \frac{k}{\rho} \right\rceil
	&\geq \sum_{k = 1}^{\lfloor n \rho \rfloor} \frac{k}{\rho} \notag \\
	&=\frac{1}{\rho} \frac{\lfloor n \rho \rfloor(\lfloor n \rho \rfloor +1)}{2} \notag\\
	&\geq \frac{(n\rho - 1)n\rho}{2\rho} \notag \\
	&=\frac{n^2 \rho}{2}\left( 1 - \frac{1}{n\rho}\right). \label{eq:lb-sum-skinny}
\end{align}
Combining~\eqref{eq:burned-skinny} with~\eqref{eq:lb-sum-skinny} gives that
\begin{align}
|B_n|
	&\leq \frac{(n+2)(n+1)}{2} + \lfloor n\rho \rfloor (n+1) - \frac{n^2 \rho}{2}\left( 1 - \frac{1}{n\rho}\right) \notag\\
	&\leq \frac{(n+2)(n+1)}{2} + \frac{n^2 \rho}{2}\left( 1 + \frac{2}{n} + \frac{1}{n \rho}\right). \label{eq:ub-skinny}
\end{align}
Similarly, combining~\eqref{eq:burned-skinny} with~\eqref{eq:ub-sum-skinny} gives
\begin{align}
|B_n|
	&\geq \frac{(n+2)(n+1)}{2} + \lfloor n\rho \rfloor (n+1) -  \frac{n^2 \rho}{2} \left(1 + \frac{1}{n \rho} + \frac{2}{n}\right) \notag\\
	&\geq \frac{(n+1)^2}{2} + (n \rho-1)n -\frac{n^2 \rho}{2} \left(1 + \frac{1}{n \rho} + \frac{2}{n}\right) \notag\\
	&\geq \frac{(n+1)^2}{2} + \frac{n^2 \rho}{2}\left(2 - \frac{2}{n \rho} - 1 - \frac{1}{n \rho} - \frac{2}{n} \right) \notag\\
	&=\frac{(n+1)^2}{2} + \frac{n^2 \rho}{2}\left(1 - \frac{3}{n \rho} - \frac{2}{n}\right). \label{eq:lb-skinny}
\end{align}
The inequalities~\eqref{eq:ub-skinny} and~\eqref{eq:lb-skinny} together show that
\[
\lim_{n \to \infty} \frac{|B_n|}{|V(Q_n)|} = \lim_{n \to \infty} \frac{|B_n|}{(n+1)^2} = \frac{1}{2} + \frac{\rho}{2},
\]
and the proof follows.\end{proof}

We need one more lemma before we prove Theorem~\ref{thm:growth-by-one}.

\begin{lemma}\label{lem:layer-cake}
Let $\underline{Q}$ be the sequence of square grids such that $Q_n = [0,n]^2$ and let $k$ be an arbitrary positive integer. Let $\underline{v} = (v_n)_{n \geq 0}$ be a sequence of activated vertices such that there are never more than $k$ consecutive $\bullet$'s and with the property that if $v_n \neq \bullet$, then $v_n = (n, n)$. For such a sequence $\underline{v}$, we have that $\delta (\underline{Q},\underline{v}) =1$.
\end{lemma}

\begin{proof}
For such a sequence $\underline{v}$, at a given time $n$, the unburned region in the graph $Q_n$ consists of  `triangular regions' along the top and right edges and a region contained in a ball centred at the top-righthand corner.

Along the top edge of the grid, each of these triangular regions consist of at most $k$ unburned vertices in the top-most row, above at most $k-2$ consecutive unburned vertices, and continuing in this way to a row of either $1$ or $2$ consecutive unburned vertices.   These unburned triangular regions are contained inside an $n \times \lceil k/2\rceil$ grid. The same applies to the right edge of the grid.

It will take up to $k$ steps until the burned region around a vertex activated at a particular moment will be in the same connected component of the subgraph of burned vertices as the origin. Hence, there may be some unburned vertices in a ball of radius $2k-1$ around the vertex in the top-righthand corner.  All these vertices are either within the top $k$ rows of the grid, or within the right-most $k$ columns of the grid.

Estimating very roughly, all the unburned vertices are either contained in a $n \times k$ grid along the top edge or a $k \times n$ grid along the right-most edge.  These two regions contain at most $2nk$ vertices.

Thus, at time step $n$, the number of unburned vertices is at most $2nk$.  Therefore,
\[
|B_n| \geq (n+1)^2 -   2nk
\]
and so
 \begin{align*}
 \lim_{n \to \infty} \frac{|B_n|}{(n+1)^2}
  &= 1 - \lim_{n \to \infty} \frac{2nk}{(n+1)^2}\\
   &= 1-0 = 1,
\end{align*}
as $k$ is fixed.
\end{proof}

Together, the activated sequences given by each of Lemma~\ref{lem:density0}, Lemma~\ref{lem:skinny-triangle}, and Lemma~\ref{lem:layer-cake} can be used to prove Theorem~\ref{thm:growth-by-c} in the case $c = 1$, which we restate here.

\begin{theorem}\label{thm:growth-by-one}
Let $\underline{S}$ be the sequence of square grids such that $S_n = [-n,n]^2$. For each $\epsilon \in [\frac{1}{2},1]$, there is a valid sequence of activated vertices $\underline{v}=(v_1,v_2,\dots)$ such that the burning density in $\underline{S}$ is $\delta (\underline{S},\underline{v})=\epsilon$.
\end{theorem}

\begin{proof}
The construction of sequences of activated vertices is given by choosing separate sequences for each of the four quadrants: $[0, n]^2$, $[-n, 0] \times [0, n]$, $[-n, 0]^2$, and $[0, n] \times [-n, 0]$.  The sequences are chosen so that at each time step, exactly one quadrant has an activated vertex.  Except for the first activated vertex $v_0 = (0, 0)$, there are no other activated vertices chosen on the overlapping regions between these quadrants.  The region of overlap, $[-n, n] \times \{0\} \cup \{0\} \times [-n, n]$, has density tending to $0$ in the square grid $S_n$.  Thus, the total limiting density of the sequence $\underline{v}$ is the average density of the restriction of this sequence to each of the four quadrants.

Let $a \in \{0,1, 2, 3, 4\}$ and $0 \leq \rho < 1$ be such that $\epsilon = \frac{1}{2} + \frac{a + \rho}{8}$.  The values of $a$ and $\rho$ will determine the type of construction used.

If $\rho = 0$, then burn a vertex in each quadrant every fourth step as follows.  For $a$ of the quadrants, choose activated vertices according to Lemma~\ref{lem:layer-cake} and for the remaining quadrants, activate vertices according to Lemma \ref{lem:density0}.

If $\rho > 0$, then $a \in \{0, 1, 2, 3\}$.  We produce a sequence $\underline{u}$ (containing $\bullet$'s) using Lemma \ref{lem:skinny-triangle}, which is constrained to one quadrant and has density $\frac{1}{2} (1+\rho)$ in that quadrant. Let $\{ a_i \} _{i=0}^\infty$ be the sequence of indices $a_i$ such that $u_{a_i} = \bullet$.  It remains to define three sequences of activated vertices for the remaining quadrants so that for every $i$, a vertex is activated in exactly one of these at time $a_i$ and with the property that $a$ of the quadrants have burning density tending to $1$ and the remaining have burning density tending to $1/2$.   For every $i$, at time $a_i$, activate a vertex in the first such quadrant if $i \equiv 0 \pmod{3}$, in the second if $i \equiv 1 \pmod{3}$, and in the third, otherwise.  Since none of these three quadrants are burned on consecutive steps, the conditions of Lemma \ref{lem:density0} are satisfied for quadrants of density $1/2$. Since $\underline{u}$ contains at least one $\bullet$ every $\left \lceil (1-\rho)^{-1} \right \rceil$ steps, each of these three quadrants are burned at least once every $3 \left \lceil (1-\rho)^{-1} \right \rceil$ steps, so that the conditions of Lemma \ref{lem:layer-cake} are satisfied for quadrants of density $1$. Combining the burning sequences of each quadrant, we then have the desired sequence $\underline{v}$. \end{proof}

Next, it is shown that burning density $1$ can be realized for any growth function $f(n) = cn$, where $c$ is a positive constant.

\begin{theorem}\label{thm:growth-by-c-density1}
For any constant $c \geq 1$, let $\underline{S}$ be the sequence of square grids such that $S_n = [ - \lceil cn \rceil , \lceil cn \rceil ]^2$. Then there is a sequence of activated vertices $\underline{v}$ with burning density $\delta (\underline{S},\underline{v} ) = 1$.
\end{theorem}

\begin{proof}
The sequence of activated vertices is defined so that, in alternating time steps, the newly activated vertex is placed in one of four different regions.  The sequence of activated vertices in the region $y \geq |x|$ is defined here and the complete sequence is obtained by three rotations of these points, staggered at distinct times modulo $4$.

In the region $y \geq |x|$, the activated vertices form the subsequence $(v_{4t})_{t \geq 0}$.  At each time step $n$ that is divisible by $4$, an activated vertex is placed along the top-most edge of the grid, the line $y = \lceil cn \rceil$, with the $x$-coordinates of the activated vertices changing in the following recursive way.  Set $v_0 = (0, 0)$ and for every $t \geq 1$, given $v_{4(t-1)} = (x_{4(t-1)}, \lceil 4c(t-1) \rceil)$, define
\[
v_{4t} =
\begin{cases}
	(x_{4(t-1)} + \lceil \sqrt{t} \rceil, \lceil 4ct \rceil)	&\text{if $x_{4(t-1)} + \lceil \sqrt{t} \rceil \leq \lceil 4ct \rceil$}\\
	(x_{4(t-1)} + \lceil \sqrt{t} \rceil - 2\lceil 4ct \rceil, \lceil 4ct \rceil)	&\text{otherwise}.
\end{cases}
\]
That is, the $x$-coordinate of $v_{4t}$ is $\lceil \sqrt{t} \rceil$ more than the $x$-coordinate of $v_{4(t-1)}$, unless such a point would be outside the region $y \geq |x|$, in which case the gap of $\lceil \sqrt{t} \rceil$ `wraps around' to the left-hand boundary of the region.

It is shown in this proof that there is a constant $C_0$ so that for any time $N$, that is sufficiently large, any point in the region defined by the equations $y \geq |x| + C_0 \sqrt{N}$ and $y \leq \lceil c N \rceil - C_0 N^{2/3}$ is burned.  As this smaller region contains all but at most $2C_0 c N \sqrt{N} + C_0 cN\cdot N^{2/3} = o(N^2)$ vertices from the entire region of the graph defined by $y \geq |x|$ and $y \leq \lceil cN \rceil$, this shows that the burning density of the sequence in this triangular region is $1$.

Set $C_0 = 60 c^2$ and fix any vertex $P = (a, b)$ in the region defined by $y \geq |x| + C_0 \sqrt{N}$ and $y \leq \lceil cN \rceil - C_0 N^{2/3}$.  If $b \leq N/2$, then $d(P, v_0) \leq 2(N/2) = N$ and so at time $N$, the vertex $P$ is burned by the first activated vertex $v_0$.

Therefore, suppose that $b > N/2$.  We will show that there is an activated vertex with $x$-coordinate within $\sqrt{N}$ of $a$ that is close enough to $P$ to burn it by time $N$.

Set $t_ 0 = \lfloor b/4c \rfloor$ so that the activated vertex $v_{4t_0}$ is the one with the largest $y$-coordinate that is at most $b$.

Consider the vertices $v_{4t}$ with $t_0 - 15c \sqrt{t_0} + 1 \leq t \leq t_0$.  Note that since $b > N/2$, then for $N$ large enough, $t_0 - 15 c \sqrt{t_0} + 1 \geq 1$.  For any such $t$, the region $y \geq |x|$ contains the point $(a, \lceil 4ct \rceil)$ (with the same $y$-coordinate as $v_{4t}$) since
\begin{align*}
|a| &\leq  b - 60c^2 \sqrt{N}\\
	&\leq b- 60 c^2 \sqrt{t_0}\\
	&\leq 4c(t_0 + 1) - 60 c^2 \sqrt{t_0}\\
	&=4c(t_0 - 15c\sqrt{t_0} + 1)\\
	&\leq 4ct \leq \lceil 4ct \rceil.
\end{align*}

Setting $t_1 = \lfloor t_0 - 15 c \sqrt{t_0} + 1 \rfloor$, the horizontal gaps, measured from left to right cyclically, between the vertices $v_{4t_1}, v_{4(t_1+1)}, \ldots, v_{4t_0}$ are, respectively, $\lceil \sqrt{t_1 + 1} \rceil, \lceil \sqrt{t_1 + 2} \rceil, \ldots, \lceil \sqrt{t_0} \rceil$.  Since the sum of these gaps is
\begin{align*}
\lceil \sqrt{t_1 + 1} \rceil + &\lceil \sqrt{t_1 + 2} \rceil + \ldots + \lceil \sqrt{t_0} \rceil
	\geq \int_{t_1}^{t_0} \sqrt{x}\ dx \geq \int_{t_0 - 15c\sqrt{t_0} + 1}^{t_0} \sqrt{x}\ dx\\
	&=\frac{2}{3}\left( t_0^{3/2} - (t_0 - 15c\sqrt{t_0} + 1)^{3/2} \right)\\
	&\geq \frac{2}{3}\left( t_0^{3/2} - \sqrt{t_0}(t_0 - 15c\sqrt{t_0} + 1) \right)\\
	& = 10ct_0 - \frac{2}{3}\sqrt{t_0}\\
	&\geq 8ct_0 + 2 = 2(4ct_0 + 1) \geq 2\lceil 4ct_0 \rceil ,
\end{align*}
which is the maximum width of any horizontal line within the region containing $v_{4t}$ with $t \in [t_0 - 15 c \sqrt{t_0} + 1, t_0]$.  Note that the final line of the above inequality follows as long as $N$, and hence $t_0$, is sufficiently large.

Since each of the horizontal gaps between consecutive activated vertices in this region are all at most $\sqrt{t_0}$, there is some $T$ with $t_0 - 15c\sqrt{t_0} + 1 \leq T \leq t_0$ so that if $v_{4T} = (x_{4T}, \lceil 4cT\rceil)$, then $|x_{4T} - a| \leq \sqrt{t_0}$.  For this $T$, we have that
\begin{align*}
d(P, v_{4T})
	&=|x_{4T}-a| + |\lceil 4cT \rceil - b|\\
	&\leq \sqrt{t_0} + (4c(t_0+1) - 4c(t_0 - 15c_0\sqrt{t_0} + 1)) = (1+60c^2)\sqrt{t_0}.
\end{align*}
Thus, $P$ will be burned by the activated vertex $v_{4T}$ at time
\begin{align*}
4T + d(P, v_{4T})
	&\leq 4t_0 + (1+60c^2)\sqrt{t_0}\\
	&\leq \frac{1}{c}\left( \lceil cN \rceil - 60c^2 N^{2/3} \right) + (1+60c^2)\sqrt{N}\\
	&\leq N + 1 - 60 c N^{2/3} + (1+60c^2)\sqrt{N}  < N.
\end{align*}
Thus, at time $N$, $P$ is burned.

Therefore, at time $N$, every vertex in the quadrant is burned, except for possibly some in the three boundary strips of width at most $C_0cN^{2/3}$.  The number of unburned vertices in a quadrant is at most $3C_0 cN^{5/3}$ and so, considering all four regions and activating one new vertex in each time step gives at most $12 C_0 cN^{5/3}$ unburned vertices and so
\[
\frac{|B_N|}{|\left[-\lceil cN \rceil, \lceil cN \rceil \right]^2 |} \geq 1 - \frac{12 c C_0 N^{5/3}}{(2\lceil cN \rceil + 1)^2} = 1 - o(1).
\]
Thus, the limiting burning density of the sequence $(v_t)_{t \geq 0}$ in the sequence of grids $([-\lceil cN \rceil, \lceil cN \rceil]^2)_{N \geq 0}$ is $1$.
\end{proof}

Theorem~\ref{thm:growth-by-c-density1}, in combination with the fact that every density in $[\frac{1}{2},1]$ can be achieved when $c=1$, shows that for any constant $c > 1$, any density in the interval $[1/2c^2,1]$ is attainable in the sequence of grids $( [ - \lceil cn \rceil , \lceil cn \rceil ]^2)_{n \geq 1}$.  Indeed, for any density in $\left[ \frac{1}{2c^2}, \frac{1}{c^2}\right]$, chose activated vertices only inside the central square grids $([-n, n]^2)_{n \geq 1}$, according to Theorem~\ref{thm:growth-by-one}.  At every time $n$, all of these vertices and their burned neighbours will themselves be inside the grid $[-n, n]^2$.  To achieve any density $\rho \in \left[ \frac{1}{c^2}, 1 \right]$, set $d = c\sqrt{\rho}$ and apply Theorem~\ref{thm:growth-by-c-density1} to the central sequence of square grids $([-\lceil dn \rceil, \lceil dn \rceil]^2)_{n \geq 1}$.  This completes the proof of Theorem~\ref{thm:growth-by-c}.

\section{Fast growing grids}\label{sec:fast-growth}

In this section, the proof of Theorem~\ref{thm:growth-by-rootn} is given. Throughout the following proofs, multiple distances in the plane are used.  Let $d_p (P,Q)$ denote $L_p$ distance between points $P$ and $Q$.  By the Cauchy-Schwarz inequality, $d_1(P, Q) \leq \sqrt{2} d_2(P, Q)$.  Also, for $P \in \mathbb{R} ^2$, let $P^*$ be the point in $\mathbb{Z} ^2$ such that $\| P^* \| _{1} \leq \| P \| _{1}$ and $d_1 (P,P^*)$ is minimum; this is essentially the nearest vertex between $P$ and the origin.  By the triangle inequality, $d_1(P, Q^*) \leq d_1(P, Q) + 2$.

Polar coordinates are used throughout the proof of Theorem~\ref{thm:growth-by-rootn}.  In polar coordinates, the pair $(r, \theta)$ denotes the point at distance $r$ from the origin, at angle $\theta$ from the positive $x$-axis.

For reference, we re-state Theorem~\ref{thm:growth-by-rootn} before we proceed to its proof.

\begin{thm-rootn}
For any positive constant $c$, let $\underline{S}$ be the sequence of square grids such that $S_n = \left [ - \lceil cn^{3/2} \rceil, \lceil cn^{3/2} \rceil \right ]^2$. Then there is a sequence of activated vertices $\underline{v}$ with lower burning density $\underline{\delta} (\underline{S},\underline{v}) >0$.
\end{thm-rootn}

\begin{proof}
The goal is to show that, with a suitable choice of activated vertices, for sufficiently large $N$ there is a disk of radius $\Omega (N^{3/2})$ in the set of burned vertices at time $N$, denoted by $B_N$. Points close to the origin will be burned by the first activated vertex (that is, the origin), while the vast majority will be burned by another activated vertex within distance $O(N)$.

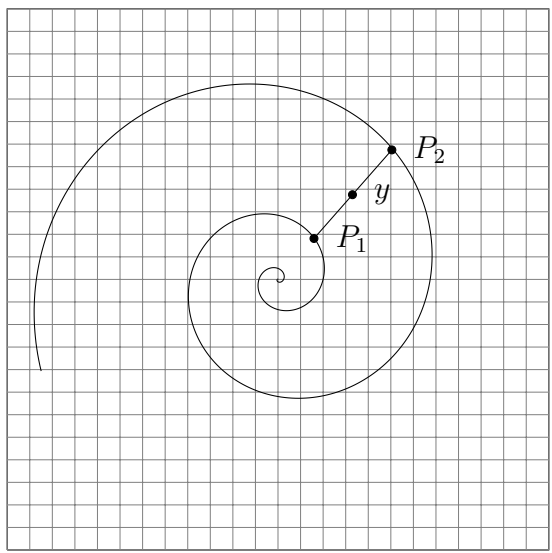
\begin{figure}[htb]
\begin{center}
\begin{tikzpicture}[scale = 0.3]
\draw (-12, -12) -- (12, -12) -- (12, 12) -- (-12, 12) -- (-12, -12);
\draw[help lines] (-12, -12) grid (12, 12);
\draw[domain=0:500,samples=500] plot ({deg(sqrt(\x))}:{0.001*(\x)^(3/2)});
\node[label=right:{$P_1$}] (p1) at ({deg(sqrt(180))}:2.415) {};
\draw[fill=black] ({deg(sqrt(180))}:2.415) circle (5pt);

\node[label=right:{$P_2$}] (p2) at ({deg(sqrt(180))}:7.645) {};
\draw[fill=black] ({deg(sqrt(180))}:7.645) circle (5pt);

\draw ({deg(sqrt(180))}:2.415) -- ({deg(sqrt(180))}:7.645) ;

\node[label=right:{$y$}] at ({deg(sqrt(180))}:5) {};
\draw[fill=black] ({deg(sqrt(180))}:5) circle (5pt);
\end{tikzpicture}
\end{center}
\caption{Spiral of activated vertices.}
\label{fig:spiral-points}
\end{figure}

For all $n \geq 0$, define $v_n = \left( cn^{3/2},\sqrt{n} \right)^*$, in polar coordinates, and let $\underline{v} = (v_n)_{n \geq 0}$ be the sequence of activated vertices.  Assume that $N \geq 600c$.  It is shown that, given the sequence $\underline{v}$, at time $N$, a positive fraction of the vertices, not depending on $N$, in the grid $S_N$ are burned.

Let $y=(d,\theta)$ be an arbitrary vertex in $S_N$, written in polar coordinates.  If $d < 400 c$, then $d_1(y, \underline{0}) \leq \sqrt{2}d < 600c \leq N$ and so $y$ is contained in the ball of burned vertices centred at the origin.

Suppose now that $y = (d, \theta)$ with $d \geq 400c$.  To find a point in $\underline{v}$ that is close to $y$, we will use the spiral defined in polar coordinates by $\left( ct^{3/2},\sqrt{t} \right)$ for reference. Let $P_1= \left(c{t_1}^{3/2},\sqrt{t_1} \right)$ and $P_2= \left(c{t_2}^{3/2},\sqrt{t_2} \right)$ be the points on this spiral that satisfy $\theta \equiv \sqrt{t_1} \equiv \sqrt{t_2} \pmod{2 \pi}$ and $ct_1^{3/2} \leq d < ct_2^{3/2}$, as in Figure~\ref{fig:spiral-points}.  Without loss of generality, assume that $\sqrt{t_1} = \theta$ and $\sqrt{t_2} = \theta + 2\pi$.

Since $d \geq 400c$, then $400c < ct_2^{3/2} = c(\sqrt{t_1} + 2\pi)^3$ and so $t_1 \geq (400^{1/3} - 2\pi)^2 > 1$. Furthermore, since $\theta =\sqrt{t_1}$, we have that $P_1 =(c \theta ^3,\theta)$ and $P_2 =(c (\theta +2 \pi) ^3,\theta +2 \pi)$.  Thus,
\begin{align}
d_2 (y,P_1)
	& \leq d_2 (P_1,P_2)  \notag\\
	&= c(\theta +2 \pi)^3 -c \theta^3 \notag\\
	&=c\theta^3 + 6c\theta^2 \pi+ 12c\theta\pi^2 + 8c\pi^3 - c \theta^3 \notag\\
	&= 6c \theta^2 \pi+12c \theta \pi^2+8c \pi^3 \notag\\
	&=6c \pi t_1+12c \pi^2 \sqrt{t_1}+8c \pi^3  \notag\\
	& \leq 386c t_1 \leq 386c \left(\frac{d}{c} \right)^{2/3}, \label{eq:dist-y-p1}
\end{align}
where the fourth equality follows since $\theta = \sqrt{t_1}$, and the inequalities in the final line follow since $1 \leq t_1 \leq (d/c)^{2/3}$.

Now let $P_3 = \left( c \lfloor t_1 \rfloor ^{3/2},\sqrt{\lfloor t_1 \rfloor} \right)$ so that $P_3^* = v_{\lfloor t_1 \rfloor}$.  The distance between $P_1$ and $P_3$ is bounded by the sum of length of the arc of the circle centred at the origin between the point $P_1$ and the point $(c t_1^{3/2}, \sqrt{\lfloor t_1 \rfloor})$ and the distance between $P_3$ and this point.  We derive that
\begin{align}
d_2 (P_1,P_3)
&\leq ct_1^{3/2} \left(\sqrt{t_1}-\sqrt{\lfloor{t_1}\rfloor} \right) + (ct_1^{3/2} - c \lfloor t_1 \rfloor^{3/2} ) \notag \\
 &\leq ct_1^{3/2} \left( \sqrt{t_1}-\sqrt{t_1 - 1}\right)  + (ct_1^{3/2} - c(t_1-1)^{3/2} ) \notag \\
&\leq c \frac{t_1^{3/2}}{\sqrt{t_1}} + \frac{3c}{2} \sqrt{t_1}
\leq c\left(1 + \frac{3}{2}\right) t_1
\leq \frac{5c}{2} \left(\frac{d}{c} \right)^{2/3}  , \label{eq:dist-p1-p3}
\end{align}
where the final inequality follows since $1 \leq t_1 \leq (d/c)^{2/3}$.

These estimates are now used to bound the $L_1$ distances between the points $y, P_1^*$, and $P_3^* = v_{\lfloor t_1 \rfloor}$.  Indeed, the distance between $y$ and $P_1^*$ is
\begin{align}
d_1(y, P_1^*)
	&\leq d_1(y, P_1) + 2 \notag\\
	&\leq \sqrt{2}d_2(y, P_1) + 2 \notag\\
	&\leq 386c \sqrt{2} \left(\frac{d}{c} \right)^{2/3} +2 \label{eq:dist-p1*-p2*}\\
	&\leq 546c\left(\frac{d}{c}\right)^{2/3} + 2, \notag
\end{align}
where the third inequality follows by \eqref{eq:dist-y-p1}.

Analogously, we find that
\begin{align}
d_1 (P_1^*,v_{\lfloor t_1 \rfloor}) & = d_1(P_1^*, P_3^*) \notag\\
	&\leq d_1(P_1, P_3) + 4\\
	&\leq \sqrt{2} d_2(P_1, P_3)  + 4	\notag\\
	&\leq c\sqrt{2}\frac{5}{2} \left(\frac{d}{c} \right)^{2/3} +4  \notag\\
	&\leq 4c  \left(\frac{d}{c} \right)^{2/3} +4, \label{eq:dist-p1*-p3*}
\end{align}
where the third inequality follows by \eqref{eq:dist-p1-p3}.

Summing the bounds in Equations~\eqref{eq:dist-p1*-p2*} and \eqref{eq:dist-p1*-p3*} gives
\begin{align*}
d_1 (y,v_{\lfloor t_1 \rfloor})
& \leq d_1 (y, P_1^*) + d_1 (P_1^*,v_{\lfloor t_1 \rfloor}) \\
& \leq 550c \left(\frac{d}{c} \right)^{2/3} +6.
\end{align*}
Since $v_{\lfloor t_1 \rfloor}$ was activated by time $\left(\frac{d}{c} \right)^{2/3}$, then $y \in B_N$ if $d_1(v_{\lfloor t_1 \rfloor}, y) \leq N - \left(\frac{d}{c} \right)^{2/3}$, which occurs if
\[
550c \left(\frac{d}{c} \right)^{2/3} +6 \leq N - \left(\frac{d}{c} \right)^{2/3}
 \]
 or equivalently if $d \leq c \left(\frac{N-6}{550c +1}\right)^{3/2}$.  Thus, for $N \geq 600c$, $B_N$ contains all points within $L_2$ distance $c \left( \frac{N-6}{600c +1} \right) ^{3/2}$ from the origin, and so $B_N$ also contains all points at $L_1$ distance at most $c \left( \frac{N-6}{600c +1} \right) ^{3/2}$ from the origin.  Therefore, $\underline{\delta} (\underline{v},\underline{x}) \geq  \frac{1}{8}(600c  +1)^{-3}$. \end{proof}

The construction given in this proof can also be used to show that if either $f(n) = \Theta(n^{3/2})$ or $f(n+1) - f(n) = O(\sqrt{n})$, then positive burning densities are possible in the sequence of graphs $([-f(n), f(n)])_{n \geq 0}$.  However, it need not be true that if $f(n) = O(n^{3/2})$ that positive burning densities are achievable.  Some such examples are detailed in Section~\ref{sec:open}.

%
%
%
%

\section{Higher dimensional grids}\label{sec:d-dim}

In this section, it is shown that there is a threshold result generalizing Theorem~\ref{thm:growth-by-rootn} to positive burning densities in arbitrary dimension grids, depending on their speed of growth.  First it is shown that faster growing grids have burning density $0$.  Theorem~\ref{thm:fast-growth} is restated for convenience.

\begin{thm-fast}
For any $d \geq 2$, let $f:\mathbb{N}_0 \to \mathbb{N}_0$ be a function with $f(n) = \omega(n^{(d+1)/d})$.  For every $n \geq 0$, let $S_n = [-f(n), f(n)]^d$, and let $\underline{S} = (S_n)_{n \geq 0}$.  Then, for every sequence of activated vertices $\underline{v}$ the burning density is $\delta(\underline{S}, \underline{v}) = 0$.
\end{thm-fast}

\begin{proof}
To prove the theorem, some bounds on the cardinalities of $d$-dimensional balls are needed.  For any $k \geq 0$, the number of solutions to $\sum_{i = 1}^d v_i = k$ with $v_i \in [0, k]$ is $\binom{k + d-1}{d-1}$ and so the number of solutions to $\sum_{i = 1}^d v_i \leq k$ is
\[
\sum_{i = 0}^{k} \binom{i+d-1}{d-1} = \binom{k+d}{d}.
\]
Therefore, considering the possible signs in the case when $v_i \in [-k, k]$,
\begin{equation}\label{eq:ub-L1-ball}
|B_1(\underline{0}, k)| \leq 2^d \binom{k+d}{d}.
\end{equation}

Let $\underline{v}$ be any sequence of activated vertices in the sequence of graphs $(S_n)_{n \geq 0}$.  For any $n$, $B_n$ is bounded by
\begin{align}
|B_n|
	&\leq \sum_{k = 0}^n |B_1(v_k, n-k)| \notag\\
	&\leq \sum_{k = 0}^n 2^d \binom{k+d}{d} \notag\\
	&= 2^d \binom{n+d+1}{d+1} \notag\\
	&\leq 2^d (n+d+1)^{d+1}. \label{eq:d-ball-bd}
\end{align}

Therefore, the burning density satisfies
\begin{align*}
\delta(\underline{S}, \underline{v})
	& = \lim_{n \to \infty} \frac{|B_n|}{(2f(n) + 1)^d}\\
	&\leq \lim_{n \to \infty} \frac{2^d (n+d+1)^{d+1}}{2^d f(n)^d}	\\
	&=\lim_{n \to \infty} \frac{(n+d+1)^{d+1}}{\omega(n^{d+1})}		\\
	&=0,
\end{align*}
where the first inequality follows by \eqref{eq:d-ball-bd}, and the second equality follows by the definition of $f(n)$.

Thus, for any sequence $\underline{v}$, $\delta(\underline{S}, \underline{v}) = 0$. \end{proof}

What remains is to prove the corresponding result for positive densities in slower growing grids.  For the proofs to come, some estimates on the cardinalities of balls of a given radius in $\mathbb{Z}^d$, their vertex boundaries, and their intersections with other balls are needed.  These are given first in Lemmas~\ref{lem:d-ball-vol-area} and \ref{lem:d-ball-sphere-intersec} before proceeding with the proof of the next threshold result, Theorem~\ref{thm:d-dim-pos-dens}.  For any point $P \in \mathbb{Z}^d$ and integer $x$, let $S_1(P, x)$ be the set of points at $L_1$-distance exactly $x$ from $P$; the sphere of radius $x$.

\begin{lemma}\label{lem:d-ball-vol-area}
For any $d \geq 2$ and $r \geq d+1$, the number of vertices in an $L_1$ ball of radius $r$ satisfies
\[
\frac{|B_1(\underline{0}, r)|}{|[-r, r]^d|} \geq \frac{1}{2^d d^d}.
\]
Furthermore, the number of vertices at $L_1$ distance exactly $r$ from the origin is bounded above and below as
\[
\frac{2^d}{(d-1)!}(r-d+1)^{d-1} \leq |S_1(\underline{0}, r)| \leq \frac{2^d}{(d-1)!}(r+d-1)^{d-1}.
\]
\end{lemma}

\begin{proof}
Since a vertex $v=(v_1,v_2,\dots,v_d)$ is in the $L_1$ ball if $\sum |v_i| \leq r$, the ball contains all vertices such that $|v_i| \leq \left \lfloor \frac{r}{d} \right \rfloor$ for all $1 \leq i \leq d$. Thus, the ball contains a cube of width $2 \left \lfloor \frac{r}{d} \right \rfloor +1$, which has volume
\[
\left( 2 \left \lfloor \frac{r}{d} \right \rfloor +1 \right) ^d \geq \left( 2 \frac{r}{d} -1 \right) ^d.
\]
Since the volume of a $d$-dimensional cube of width $2r+1$ is $(2r+1)^d$, for $r \geq d+1$, the ball satisfies
\[
\frac{|B_1(\underline{0}, r)|}{(2r+1)^d} \geq \frac{1}{d^d}\left(\frac{2r-d}{2r+1}\right)^d \geq \frac{1}{d^d 2^d}.
\]

As described in the proof of Theorem~\ref{thm:fast-growth}, the number of vertices in $S_1(\underline{0}, r)$ is the number of integer solutions to $\sum_{i = 1}^d |v_i|=r$ with $v_i \in [-r, r]$.  As before, the upper bound is
\[
|S_1(\underline{0}, r)| \leq 2^d \binom{r+d-1}{d-1} \leq \frac{2^d}{(d-1)!}(r+d-1)^{d-1}.
\]
The lower bound is obtained by disregarding the solutions with some $v_i = 0$, which gives,
\[
|S_1(\underline{0}, r)| \geq 2^d \binom{r-1}{d-1} \geq \frac{2^d}{(d-1)!}(r-d+1)^{d-1}.
\]
\end{proof}

We need the following lemma.

\begin{lemma}\label{lem:d-ball-sphere-intersec}
Let $S= S_1 (\underline{0},x)$, let $P$ be a vertex at distance $r$ from the origin and let $B= B_1 (P,y)$ be the $L_1$ ball of radius $y$ centered at $P$.  If $r \leq x+y$ and $r+x \geq y$, then
\[
|S \cap B| \geq \frac{1}{(d-1)!} \left( \frac{1}{2} (x+y-r) -d \right) ^{d-1}
\]
\end{lemma}

\begin{proof}
The proof is given by showing that the set $B_1 (\underline{0},x) \cap B$ contains a ball of radius $\left \lfloor \frac{1}{2} (x+y-r) \right \rfloor$, one face of which is contained in the set $S$. Toward this end, let $l$ be a path of length $r$ from $\underline{0}$ to $P$, and let $M$ be the vertex on $l$ at distance $\left \lceil \frac{1}{2} (r+x-y) \right \rceil$ from $\underline{0}$. 

Let $M'$ be a point on the same path with $d(\underline{0}, M') = x$ so that $M' \in S$.  Since $x, r$, and $y$ are all integers, $x+r-y \equiv x-r+y \pmod{2}$.  Since $r \leq x+y$, then $x \geq \lceil \frac{1}{2}(x+r-y)\rceil$ and so $d(M, M') = x - \lceil \frac{1}{2}(x+r-y)\rceil = \lfloor \frac{1}{2} (x-r+y) \rfloor$.  Also, since $d(M', P) = r-x \leq y$, then $M' \in B \cap S$.

Furthermore, since $x \leq r+y$, then $r \geq \lceil \frac{1}{2}(x+r-y)\rceil$ and so $d(M, P) = r - \lceil \frac{1}{2}(x+r-y)\rceil = \lfloor \frac{1}{2}(r-x+y)\rfloor$.

For any point $Q$ with $d(M, Q) \leq \lfloor \frac{1}{2}(x-r+y)\rfloor$, then
\[
d(\underline{0}, Q) \leq d(\underline{0}, M) + d(M, Q) \leq \left\lceil \frac{1}{2}(x+r-y) \right\rceil + \left\lfloor\frac{1}{2}(x-r+y)\right\rfloor = x
\]
and
\[
d(Q, P) \leq d(Q, M) + d(M, P) \leq \left\lfloor \frac{1}{2}(x-r+y) \right\rfloor + \left\lfloor \frac{1}{2}(r-x+y)\right\rfloor \leq y.
\]

It follows that $B_1 \left( M, \left \lfloor \frac{1}{2} (x+y-r) \right \rfloor \right) \subset B_1 (\underline{0},x) \cap B$. The vertex $M'$ is contained in a `face' of the ball $B_1 \left( M, \left \lfloor \frac{1}{2} (x+y-r) \right \rfloor \right)$ and of $S$.  Furthermore, any other vertex in the same face of this ball is also contained in the set $S$.  Since there are $2^d$ different faces of the ball, by Lemma~\ref{lem:d-ball-vol-area}
\begin{align*}
|S \cap B| &\geq \frac{1}{2^d} \frac{2^d}{(d-1)!} \left( \left \lfloor \frac{1}{2} (x+y-r) \right \rfloor -d+1 \right) ^{d-1} \\
&\geq \frac{1}{(d-1)!} \left( \frac{1}{2} (x+y-r) -d \right) ^{d-1},
 \end{align*}
and the proof follows. \end{proof}

Lemmas~\ref{lem:d-ball-vol-area} and \ref{lem:d-ball-sphere-intersec} are now used to prove that in the case that the grids do not grow too quickly, a random selection of activated vertices within particular sets gives a proportion of burning vertices bounded away from $0$ at any fixed time.  These random sequences are combined to give an infinite sequence with positive lower burning density.

\begin{theorem}\label{thm:d-dim-pos-dens}
For any integer $d \geq 2$, let $\underline{S}$ be the sequence of $d$-dimensional grids such that $S_n = \left [ - \lceil n^{(d+1)/d} \rceil, \lceil n^{(d+1)/d} \rceil \right ]^d$. Then there is a sequence of activated vertices $\underline{v}$ such that $\underline{\delta} (\underline{S},\underline{v}) >0$.
\end{theorem}

\begin{proof}
Let $N \geq \max\left\{2\left(\frac{3d}{d+1}\right)^d, 2^{d+2}, 16d\right\}$ be large and consider the sequence of grids $(S_n)_{n = 1}^N$.  Vertices are activated at random in the following way: for every time-step $n \in \left [ \lfloor N/6 \rfloor +1, \lfloor N/2 \rfloor \right ]$, choose one vertex at distance $\left \lfloor n^{(d+1)/d} \right \rfloor$ from the origin uniformly at random to be the vertex $v_n$ activated at time $n$ and make this choice independently of all others. Note that for $n \geq 1$, $(n+1)^{(d+1)/d} \geq n^{(d+1)/d} +1$ and so these chosen vertices are all distinct.

We will then show that there is a constant $\lambda_d >0$, depending only on $d$, such that with positive probability the fraction of burning vertices in $S_N$ at time $N$ is at least $\lambda_d$.  While the proof that such a constant $\lambda_d>0$ exists is probabilistic, we then deterministically concatenate infinitely many such sequences of activated vertices to obtain the sequence of activated vertices for $(S_n)_{n \geq 1}$.  This sequence will have \emph{upper} burning density at least $\lambda_d$ and, by choosing the `time scales' appropriately, we will show that the lower burning density is still at least $\lambda_d / 3^{d+1}$.

We begin by computing the expected number of burned vertices whose distance from the origin is in the range $\left [ (N/4)^{(d+1)/d},(N/2)^{(d+1)/d} \right ]$. For any $(N/4)^{(d+1)/d} \leq r \leq (N/2)^{(d+1)/d}$, let $P$ be any vertex at distance $r$ from the origin and consider the probability that $P$ is burned by time $N$.  In fact, only activated vertices $v_n$ with $n \leq r^{d/(d+1)}$ and $n \geq r^{d/(d+1)} - \frac{2^{1/d} d}{4(d+1)} N^{1-1/d}+1$ are considered as sources for a fire to burn $P$. Note that for this choice of $n$,
\begin{align}
\lfloor n^{(d+1)/d} \rfloor
	&\geq n^{(d+1)/d} - 1 \notag\\
	&\geq \left(r^{d/(d+1)} - \frac{2^{1/d} d}{4(d+1)} N^{1-1/d}+1\right)^{(d+1)/d} - 1 \notag\\
	&=r \left(1 - \frac{2^{1/d} d}{4(d+1)r^{d/(d+1)}} N^{1-1/d} + \frac{1}{r^{d/(d+1)}} \right)^{(d+1)/d} - 1 \notag\\
	&\geq r\left(1 - \frac{d+1}{d} \frac{2^{1/d} d}{4(d+1)r^{d/(d+1)}} N^{1-1/d}  + \frac{d+1}{d}\frac{1}{r^{d/(d+1)}} \right) - 1 \notag\\
	&= r - \frac{2^{1/d}}{4} N^{1 - 1/d} r^{1/(d+1)} + \frac{d+1}{d} r^{1/(d+1)} - 1 \notag\\
	&\geq r -  \frac{2^{1/d}}{4} N^{1 - 1/d} \left(\frac{N}{2}\right)^{1/d} + \frac{(d+1)}{d}\left(\frac{N}{4}\right)^{1/d} - 1 \notag\\
	&\geq r - \frac{N}{4} \label{eq:n-lb}
\end{align}

For any $n \leq r^{d/(d+1)}$, $P$ is burned by $v_n$ by time $N$ if{f} $d(P, v_n) \leq N-n$.  Thus, for $n \leq r^{d/(d+1)}$, if $d(P, v_n) \leq N-r^{d/(d+1)}$, then $P$ is burned by time $N$.  Since activated vertices are chosen uniformly at random, the probability that $v_n$ burns $P$ by time $N$ is, 
\begin{multline}\label{eq:prob-random-burn}
\mathbb{P} (P \text{ is burned by time } N \text{ by } v_n ) \\
\geq \frac{\left | S_1 \left( \underline{0},\left \lfloor n^{(d+1)/d} \right \rfloor \right) \cap B_1 \left( P,N-r^{d/(d+1)} \right) \right |}{\left | S_1 \left( \underline{0},\left \lfloor n^{(d+1)/d} \right \rfloor \right) \right |}.
\end{multline}

Lemma~\ref{lem:d-ball-sphere-intersec} is used to bound the numerator of Equation~\eqref{eq:prob-random-burn}.  To verify the conditions of Lemma~\ref{lem:d-ball-sphere-intersec}, note that
\begin{align*}
\lfloor n^{(d+1)/d} \rfloor + (N-r^{d/(d+1)}) 
	&\geq r - \frac{N}{4} + N - r^{d/(d+1)}\\
	&\geq r + \frac{3N}{4} - \frac{N}{2} = r + \frac{N}{4} > r
\end{align*}
where the first inequality follows from \eqref{eq:n-lb} and the second from the upper bound for $r$.

Again using \eqref{eq:n-lb},
\begin{align*}
r+ \lfloor n^{(d+1)/d} \rfloor &\geq r + r - \frac{N}{4} = 2r - \frac{N}{4}\\
		&\geq 2\left(\frac{N}{4}\right)^{(d+1)/d} - r^{d/(d+1)}\\
		&\geq N - r^{d/(d+1)}
\end{align*}
where the second inequality follows from the lower bound on $r$ and the third inequality follows since $N \geq 2^{d+2}$.

Therefore, by Lemma~\ref{lem:d-ball-sphere-intersec}, the numerator of \eqref{eq:prob-random-burn} is
\begin{align}
| S_1 &\left( \underline{0},\left \lfloor n^{(d+1)/d} \right \rfloor \right) \cap B_1 \left( P,N-r^{d/(d+1)} \right) | \notag\\
	& \geq \frac{1}{(d-1)!}\left( \frac{1}{2}(\lfloor n^{(d+1)/d} \rfloor + N -r^{d/(d+1)} - r) - d\right)^{d-1} \notag\\
	& \geq \frac{1}{(d-1)!}\left(\frac{1}{2}(-N/4 + N - r^{d/(d+1)}) - d\right)^{d-1} \notag\\
	& \geq \frac{1}{(d-1)!}\left(\frac{1}{2}(3N/4 - N/2) -d\right)^{d-1}\notag\\
	& = \frac{1}{(d-1)!} \left(\frac{N}{8} - d\right)^{d-1} \geq \frac{1}{(d-1)!}\left(\frac{N}{16}\right)^{d-1} \label{eq:num-lb}
\end{align} 
where the second inequality follows from equation~\eqref{eq:n-lb} and the final inequality follows since $N \geq 16d$.

For the denominator of \eqref{eq:prob-random-burn}, Lemma~\ref{lem:d-ball-vol-area} is used for an upper bound
\begin{align}
|S_1(\underline{0}, \lfloor n^{(d+1)/d} \rfloor)|
	&\leq \frac{2^d}{(d-1)!} (\lfloor n^{(d+1)/d} \rfloor + d-1)^{d-1} \notag\\
	&\leq \frac{2^d}{(d-1)!} (2 \lfloor n^{(d+1)/d} \rfloor)^{d-1} \notag\\
	&\leq \frac{2^d}{(d-1)!}\left(2 (N/2)^{(d+1)/d}\right)^{d-1} \notag\\
	&= \frac{2^{d-1+1/d} N^{d-1/d}}{(d-1)!}. \label{eq:denom-ub}
\end{align}

Combining Equations~\eqref{eq:prob-random-burn}, \eqref{eq:num-lb}, and \eqref{eq:denom-ub} gives
\begin{equation}\label{eq:burn-prob-lb}
\mathbb{P} (P \text{ is burned by time } N \text{ by } v_n ) \geq \frac{\frac{1}{(d-1)!}(N/16)^{d-1}}{\frac{2^{d-1+1/d}}{(d-1)!}N^{d-1/d}}
	=\frac{1}{2^{5d-5+1/d} N^{1-1/d}}.
\end{equation}

Therefore, setting $l= r^{d/(d+1)} - \frac{d2^{1/d}}{4(d+1)} N^{1-1/d}$, the probability that $P$ is not burned by time $N$ by $v_n$ with $n \in [l+1, r^{d/(d+1)}]$ is
\begin{align*}
\mathbb{P} &( P \text{ is not burned by time } N)\\
 & \leq \prod_{n=l+1}^{r^{d/(d+1)}} \mathbb{P} (P \text{ is not burned by time } N \text{ by } v_n ) \\
 & \leq \prod_{n = l+1}^{r^{d/(d+1)}} \left(1 - \frac{1}{2^{5d-5+1/d} N^{1-1/d}}\right) \\
 & = \left(1 - \frac{1}{2^{5d-5+1/d} N^{1-1/d}}\right)^{r^{d/(d+1)} - l}\\
 & = \left(1 - \frac{1}{2^{5d-5+1/d} N^{1-1/d}}\right)^{ \frac{d2^{1/d}}{4(d+1)} N^{1-1/d}}\\
 & \leq \exp\left(- \frac{1}{2^{5d-5+1/d} N^{1-1/d}}\right)^{ \frac{d2^{1/d}}{4(d+1)} N^{1-1/d}}\\
 &=\exp\left(- \frac{1}{2^{5d-5+1/d} N^{1-1/d}} \cdot \frac{d2^{1/d}}{4(d+1)} N^{1-1/d} \right)\\
 &=\exp\left(-\frac{d}{(d+1) 2^{5d-3}} \right) < 1,
\end{align*}
where the second inequality follows by \eqref{eq:burn-prob-lb}.

Since this bound depends only on $d$, we have the desired $\lambda_d$. To construct an infinite sequence with positive lower density from the existing finite sequences, let $N_1 > \max\left\{2\left(\frac{3d}{d+1}\right)^d, 2^{d+2}, 16d\right\}$ and for every $i \geq 1$, set $N_{i+1} = 3 N_{i}$. For every $i \geq 1$, let $\underline{v}_i$ be a sequence of activated vertices given by the probabilistic proof above.  That is, for every $n \in [\lfloor N_i/6 \rfloor + 1, \lfloor N_i/2 \rfloor]$, $(\underline{v}_i)_n$ is a vertex in $S_{N_i}$ so that at time $N_i$, the fraction of vertices in $S_{N_i}$ burning (because of activated vertices in $\underline{v}_i$) is at least $\lambda_d$.  Construct an infinite sequence of activated vertices $\underline{v}$ by concatenating the sequences $(\underline{v}_i)_{i \geq 1}$ (selecting activated vertices arbitrarily for all times before $N_1/6$). Since
\[
\frac{N_{i+1}}{6} +1 = \frac{3N_i}{6} + 1 > \frac{N_i}{2},
 \]
there is no overlap between the sequences $\underline{v}_i$. 

The construction guarantees that the burning density is at least $\lambda_d$ at time $N_i$ for all $i$. Since $\frac{N_{i+1}}{N_i} =3$, for any $t \leq 2N_i$, the burning density in $S_{N_i + t}$ at time $N_i + t$ is at least
\[
\frac{\lambda_d (2N_i^{(d+1)/d}+1)^d}{(2(N_i+t)^{(d+1)/d} +1)^d} \geq \lambda_d \left(\frac{2N_i^{(d+1)/d} +1}{2(3N_i)^{(d+1)/d} + 1}\right)^d \geq \lambda_d \left(\frac{1}{3^{(d+1)/d}}\right)^d =  \frac{\lambda_d}{3^{d+1}}.
\]
Therefore, the minimum burning density at \emph{any} time is at least $\lambda_d/3^{d+1}$ and so $\underline{\delta}(\underline{S}, \underline{v}) \geq \lambda_d/3^{d+1}$.
\end{proof}

\section{Further directions}\label{sec:open}

One possible variation on the burning density is the following. Let $(G_0, G_1, \ldots)$ be a sequence of graphs with the property that for every $n \geq 1$, $G_{n-1}$ is an induced subgraph of $G_n$.  A sequence of activated vertices, $\underline{v}$, is a \emph{connected burning} sequence if for every $n \geq 0$ with $v_n \neq \bullet$, $v_n$ is adjacent to a vertex in $N_{G_n}[B_n]$.

In a connected burning sequence, an activated vertex is always chosen adjacent to an already burning vertex and so the set of burned vertices is always a connected subgraph. Which burning densities are achievable by connected burning sequences?   The sequences of activated vertices given by Lemmas~\ref{lem:density0} and~\ref{lem:skinny-triangle} are always connected burning sequences.  This shows, as in the proof of Theorem~\ref{thm:growth-by-one}, that in the sequence of grids $([-n, n]^2)_{n \geq 0}$, connected burning sequences can achieve any density in $[1/2, 5/8]$.  Are any other values achievable by connected burning sequences?

With regard to burning in higher dimensional grids, it would be interesting to determine a meaningful condition on grid-size growth that determines whether or not positive burning density is achievable. We have made such determinations concerning the ``threshold'' $n^{(d+1)/d}$, but what about growth functions that are not well-behaved or even strictly increasing? Certain considerations obstruct a straightforward generalization; for example, the burning density need not always exist. Indeed, take $f(n) = 2^{\lfloor \log n \rfloor}$. Then if $n = 2^k -1$, $f(n) = 2^{k-1}$ and any vertex within distance $2^k -1$ of the origin is burned at time $n$, leaving at most $4$ vertices unburned (the corners of the grid). However, if $n = 2^k$, then $f(n) = 2^k$ and at most $\left( 2^{k-1} +1 \right) ^2$ vertices are burned at time $n$. In the former case, the burning density is at least $1 - o(1)$, while in the later it is at most $\frac{1}{4} + o(1)$. Thus, the burning density does not exist.

A similar result follows when $f(n)$ is of maximum order. For this, let $f(n) = \left \lfloor \left( 2^{2^{\lfloor \log \log n \rfloor}} \right) ^{3/2} \right \rfloor$. If $n = 2^{2^k} -1$, then $f(n) = \left \lfloor \left( 2^{2^{k-1}} \right) ^{3/2} \right \rfloor$ and every vertex in the graph is burned. If $n = 2^{2^k}$, then $f(n) = \left \lfloor \left( 2^{2^k} \right) ^{3/2} \right \rfloor$ and at most $\left( 2^{2^{k-1}} +1 \right) ^3$ vertices are burned. Hence, the upper burning density is $1$, and since
\[
\lim_{k \rightarrow \infty} \left( \frac{2^{2^{k-1}}}{2^{2^k}} \right) ^3 = 0
\]
the lower burning density is $0$.

Also notice that it is not merely the order of the graphs that determines whether positive density is achievable. We showed that grids of order $\left( n^{3/2} \right) ^2 = n^3$ can yield positive lower density. However, even graphs of order $n^{2+\epsilon}$ in $\mathbb{Z} ^2$ may force density $0$. This is simply because ``thin'' grids mimic the $1$-dimensional model, in which the threshold for positive densities is $n^2$. For example, let the graph sequence be defined by $S_n = [0, n^{\varepsilon/2}] \times [0, n^{2+\varepsilon/2}]$. Then at arbitrary time $N$, we have that the number of burned vertices is at most
\[
\sum_{k=1}^N 2k n^{\varepsilon/2} \leq N^2 \cdot N^{\varepsilon/2} = o(|S_N|) .
\]
Thus, it is necessary to consider the structure of the grids comprising the graph sequence.

As with the firefighting problem, one may also study graph burning on other infinite lattices. Further, what happens to the burning density when the sequence of grids is not symmetric about the origin?

\section*{Acknowledgements}
The authors would like to thank the referee for their very careful reading of the paper and the many helpful suggestions which greatly improved a number of the proofs.


\begin{thebibliography}{}
%
%
\bibitem{benevides1}
{\sc F.~S. Benevides and M.~Przykucki}, {\em On slowly percolating sets of
  minimal size in bootstrap percolation}, Electron. J. Combin., 20 (2013),
  pp.~1--20.

\bibitem{benevides2}
{\sc F.~S. Benevides and M.~Przykucki}, {\em Maximum percolation time in
  two-dimensional bootstrap percolation}, SIAM J. Discrete Math., 29 (2015),
  pp.~224--251.

\bibitem{BBJRR17}
{\sc S.~Bessy, A.~Bonato, J.~Janssen, D.~Rautenbach, and E.~Roshanbin}, {\em
  Burning a graph is hard}, Discrete Appl. Math., 232 (2017), pp.~73--87.

\bibitem{BBJRR18}
{\sc S.~Bessy, A.~Bonato, J.~Janssen, D.~Rautenbach, and E.~Roshanbin}, {\em
  Bounds on the burning number}, Discrete Appl. Math., 235 (2018), pp.~16--22.

\bibitem{BJR14}
{\sc A.~Bonato, J.~Janssen, and E.~Roshanbin}, {\em Burning a graph as a model
  of social contagion}, Lecture Notes in Comput. Sci., 8882 (2014), pp.~13--22.

\bibitem{BJR16}
{\sc A.~Bonato, J.~Janssen, and E.~Roshanbin}, {\em How to burn a graph},
  Internet Math., 12 (2016), pp.~85--100.

\bibitem{BK}
{\sc A.~Bonato and S.~Kamali}, {\em Approximation and algorithms for burning
  graphs}.
\newblock Preprint, 2018.

\bibitem{BL}
{\sc A.~Bonato and T.~Lidbetter}, {\em Bounds on the burning numbers of spiders
  and path-forests}.
\newblock To appear.

\bibitem{BGPS06}
{\sc S.~Boyd, A.~Ghosh, B.~Prabhakar, and D.~Shah}, {\em Randomized gossip
  algorithms}, IEEE Trans. Inform. Theory, 52 (2006), pp.~2508--2530.

\bibitem{CW}
{\sc L.~Cai and W.~Wang}, {\em The surviving rate of a graph for the
  firefighter problem}, SIAM J. Discrete Math., 23 (2009), pp.~1814--1826.

\bibitem{finbow}
{\sc S.~Finbow and G.~MacGillivray}, {\em The firefighter problem: A survey of
  results, directions and questions}, Australas. J. Combin., 43 (2009),
  pp.~57--77.

\bibitem{shannon}
{\sc S.~L. Fitzpatrick and L.~Wilm}, {\em Burning circulant graphs}.
\newblock Preprint, 2018.

\bibitem{granovetter}
{\sc M.~Granovetter}, {\em Threshold models of collective behavior}, Am. J.
  Sociol.,  (1978), pp.~1420--1443.

\bibitem{HW65}
{\sc J.~M. Hammersley and D.~J.~A. Welsh}, {\em First-passage percolation,
  subadditive processes, stochastic networks, and generalized renewal theory},
  in Bernoulli 1713 Bayes 1763 Laplace 1813, J.~Neyman and L.~M. LeCam, eds.,
  Springer, Berlin, Heidelberg, 1965, pp.~61--110.

\bibitem{tH74}
{\sc T.~E. Harris}, {\em Contact interactions on a lattice}, Ann. Probab., 2
  (1974), pp.~969--988.

\bibitem{hartnell}
{\sc B.~L. Hartnell}, {\em Firefighter! an application of domination}.
\newblock Presentation at 25th Manitoba Conference on Combinatorial Mathematics
  and Computing, University of Manitoba in Winnipeg, MB, 1995.

\bibitem{Kramer14}
{\sc A.~D.~I. Kramer, J.~E. Guillory, and J.~T. Hancock}, {\em Experimental
  evidence of massive-scale emotional contagion through social networks}, Proc.
  Natl. Acad. Sci. USA, 111 (2014), pp.~8788--8790.

\bibitem{LL}
{\sc M.~Land and L.~Lu}, {\em An upper bound on the burning number of graphs},
  in Algorithms and models for the web graph, Springer, Cham, 2016, pp.~1--8.

\bibitem{rburn}
{\sc D.~Mitsche, P.~Pra\l{}at, and E.~Roshanbin}, {\em Burning graphs---a
  probabilistic perspective}, Graphs Combin., 33 (2017), pp.~449--471.

\bibitem{przykucki}
{\sc M.~Przykucki}, {\em Maximal percolation times in hypercubes under
  2-bootstrap percolation}, Electron. J. Combin., 19 (2012), pp.~1--13.

\bibitem{dR73}
{\sc D.~Richardson}, {\em Random growth in a tessellation}, Math. Proc.
  Cambridge Philos. Soc., 74 (1973), pp.~515--528.

\bibitem{thez}
{\sc E.~Roshanbin}, {\em Burning a graph as a model of social contagion}, PhD
  thesis, Dalhousie University, Halifax, NS, 2016.

\bibitem{schelling}
{\sc T.~Schelling}, {\em Micromotives and Macrobehavior}, Norton, New York, NY,
  1978.



\end{thebibliography}


\end{document}